\documentclass[12pt,leqno]{amsart}

\usepackage{amscd}
\usepackage{amsmath}
\usepackage{amsfonts}
\usepackage{amssymb}
\usepackage{amsthm}
\usepackage{hyperref}

\theoremstyle{plain}% default
\newtheorem{thm}{Theorem}[section]
\newtheorem{lem}[thm]{Lemma}

\newtheorem{cor}[thm]{Corollary}

\theoremstyle{definition}
\newtheorem{defn}{Definition}[section]

\theoremstyle{remark}

\title{Bogomolov multipliers for some $p$-groups of nilpotency class $2$}
\author{Ivo M. Michailov}
\address{Faculty of Mathematics and Informatics, Shumen University "Episkop Konstantin Preslavski", Universitetska str. 115, 9700 Shumen, Bulgaria}
\email{ivo\_michailov@yahoo.com}
\date{\today}
\thanks{This work is partially supported by a project No RD-08-241/12.03.2013 of Shumen University}
\begin{document}
\baselineskip 15pt
\begin{abstract}
The Bogomolov multiplier $B_0(G)$ of a finite group $G$ is defined
as the subgroup of the Schur multiplier consisting of the cohomology
classes vanishing after restriction to all abelian subgroups of $G$.
The triviality of the Bogomolov multiplier is an obstruction to
Noether's problem. We show that if $G$ is a central product of $G_1$
and $G_2$, regarding $K_i\leq Z(G_i), i=1,2$, and $\theta:G_1\to
G_2$ is a group homomorphism such that its restriction
$\theta\vert_{K_1}:K_1\to K_2$ is an isomorphism, then the
triviality of $B_0(G_1/K_1), B_0(G_1)$ and $B_0(G_2)$ implies the
triviality of $B_0(G)$. We give a positive answer to Noether's problem for all $2$-generator
$p$-groups of nilpotency class $2$, and for one series of
$4$-generator $p$-groups of nilpotency class $2$ (with the usual requirement for the roots of
unity).
\end{abstract}

\maketitle
\newcommand{\Gal}{{\rm Gal}}
\newcommand{\im}{{\rm im}}
\newcommand{\res}{{\rm res}}
\newcommand{\GL}{{\rm GL}}
\newcommand{\Br}{{\rm Br}}
\newcommand{\lcm}{{\rm lcm}}
\newcommand{\ord}{{\rm ord}}
\renewcommand{\thefootnote}{\fnsymbol{footnote}}
\numberwithin{equation}{section}

%--------------------------------------Section 1-------------------------------------------------
\section{Introduction}
\label{1}

Let $K$ be a field, $G$ a finite group and $V$ a faithful
representation of $G$ over $K$. Then there is a natural action of
$G$ upon the field of rational functions $K(V)$. \emph{The
rationality problem} (also known as \emph{Noether's problem} when
$G$ acts on $V$ by permutations) then asks whether the field of
$G$-invariant functions $K(V)^G$ is rational (i.e., purely
transcendental) over $K$. A question related to the above mentioned
is whether $K(V)^G$ is stably rational, that is, whether there exist
independent variables $x_1,\dots,x_r$ such that
$K(V)^G(x_1,\dots,x_r)$ becomes a purely transcendental extension of
$K$. This problem has close connection with L\"uroth's problem
\cite{Sh} and the inverse Galois problem \cite{Sa2,Sw}.

Saltman \cite{Sa2} found examples of groups $G$ of order $p^9$ such
that $\mathbb C(V)^G$ is not stably rational over $\mathbb C$. His
main method was application of the unramified cohomology group
$H_{nr}^2(\mathbb C(V)^G,\mathbb Q/\mathbb Z)$ as an obstruction.
Bogomolov \cite{Bo} proved that $H_{nr}^2(\mathbb C(V)^G,\mathbb
Q/\mathbb Z)$ is canonically isomorphic to
\begin{equation*}
B_0(G)=\bigcap_{A}\ker\{\res_G^A:H^2(G,\mathbb Q/\mathbb Z)\to
H^2(A,\mathbb Q/\mathbb Z)\}
\end{equation*}
where $A$ runs over all the bicyclic subgroups of $G$ (a group $A$
is called bicyclic if $A$ is either a cyclic group or a direct
product of two cyclic groups). The group $B_0(G)$ is a subgroup of
the Schur multiplier $H^2(G,\mathbb Q/\mathbb Z)$, and
Kunyavski\u{i} \cite{Ku} called it the \emph{Bogomolov multiplier}
of $G$. Thus the vanishing of the Bogomolov multiplier is an
obstruction to Noether's problem.

Recently, Moravec \cite{Mo1} used a notion of the nonabelian
exterior square $G\wedge G$ of a given group $G$ to obtain a new
description of $B_0(G)$. Namely, he proved that $B_0(G)$ is
(non-canonically) isomorphic to the quotient group $M(G)/M_0(G)$,
where $M(G)$ is the kernel of the commutator homomorphism $G\wedge
G\to [G,G]$, and $M_0(G)$ is the subgroup of $M(G)$ generated by all
$x\wedge y$ such that $x,y\in G$ commute. Moravec studied the
functor $B_0(G)$ in \cite{Mo1}, and in particular he found the five
term exact sequence
\begin{equation*}
B_0(E)\to B_0(E/N)\to\frac{N}{\langle \mathcal K(E)\cap N\rangle}\to
E^{ab}\to(E/N)^{ab}\to 0,
\end{equation*}
where $E$ is any group, $N$ a normal subgroup of $E$ and $\mathcal
K(E)$ denotes the set of commutators in $E$.

If we assume that $N$ is a central subgroup of $E$, we derive in
Section \ref{3} a three term exact sequence
\begin{equation}\label{e1}
B_0(E)\to B_0(E/N)\to\frac{N\cap [E,E]}{\langle \mathcal K(E)\cap
N\rangle}.
\end{equation}

Let $G$ be a central product of $G_1$ and $G_2$, regarding $K_i\leq
Z(G_i)$, where $Z(G_i)$ is the centre of $G_i$ for $i=1,2$. Kang and
Kunyavski\u{i} \cite{KK} raised the question whether the triviality
of $B_0(G_1)$ and $B_0(G_2)$ implies the triviality of $B_0(G)$.
With the aid of \eqref{e1} we prove our first main result, stating
that if $\theta:G_1\to G_2$ is a group homomorphism such that its
restriction $\theta\vert_{K_1}:K_1\to K_2$ is an isomorphism, then
the triviality of $B_0(G_1/K_1), B_0(G_1)$ and $B_0(G_2)$ implies
the triviality of $B_0(G)$.

We prove two applications of this result obtaining the triviality of
the Bogomolov multipliers for the extra-special $p$-groups and for
the $4$-generator $p$-groups that are central products of metacyclic
$p$-groups. Of course, we will not consider groups that are direct
products of smaller groups, because of the following.

\begin{thm}\label{t1.1}
{\rm (Kang }\cite[Theorem 4.1]{Ka3}{\rm )} Let $G$ and $H$ be finite
groups. Then $B_0(G\times H)$ is isomorphic to $B_0(G)\times
B_0(H)$. As a corollary, if $B_0(G)$ and $B_0(H)$ are both trivial,
then also is $B_0(G\times H)$.
\end{thm}

The Bogomolov multipliers for the groups of order $p^n$ for $n\leq
6$ were calculated recently in \cite{HoK,HKK,Mo2,CM}. The reader is
referred to the paper \cite{KK} for a survey of groups with trivial
Bogomolov multipliers.

\begin{defn}
We say that a group $G$ has the AEC (Abelian Extension of a Cyclic group) property if
$G$ has a normal abelian subgroup $H$ such that the quotient group
$G/H$ is cyclic
\end{defn}

Bogomolov \cite[Lemma 4.9]{Bo} proved that if $G$ has the AEC
property then $B_0(G)=0$. On the other hand, Noether's problem for
$p$-groups with the AEC property is still not solved entirely.
However, there are number of partial results that have been obtained
recently. We list two of them.

\begin{thm}\label{t1.2}
{\rm (Kang}\cite[Theorem 1.5]{Ka1}{\rm )} Let $G$ be a metacyclic
$p$-group with exponent $p^e$, and let $K$ be any field such that
{\rm (i)} char $K = p$, or {\rm (ii)} char $K \ne p$ and $K$
contains a primitive $p^e$-th root of unity. Then $K(G)$ is rational
over $K$.
\end{thm}

\begin{thm}\label{t1.3}
{\rm (Michailov}\cite[Theorem 1.8]{Mi1}{\rm )} Let $G$ be a group of order $p^n$ for $n\geq 2$ with an abelian
subgroup $H$ of order $p^{n-1}$, and let $G$ be of exponent $p^e$.
Choose any $\alpha\in G$ such that $\alpha$ generates $G/H$, i.e.,
$\alpha\notin H,\alpha^p\in H$. Denote $H(p)=\{h\in H: h^p=1,h\notin
H^p\}\cup\{1\}$, and assume that $[H(p),\alpha]\subset H(p)$. Denote by
$G_{(i)}=[G,G_{(i-1)}]$ the lower central series for $i\geq 1$ and
$G_{(0)}=G$. Let the $p$-th lower central subgroup $G_{(p)}$ be
trivial. Assume that {\rm (i)} char $K = p>0$, or {\rm (ii)} char $K
\ne p$ and $K$ contains a primitive $p^e$-th root of unity. Then
$K(G)$ is rational over $K$.
\end{thm}

In Section \ref{5} of the present article we will give a positive
answer to Noether's problem for all $2$-generator $p$-groups of
nilpotency class $2$ (with the usual requirement for the roots of
unity). It seems that almost all known results for Noether's problem
regarding $p$-groups actually hold only for $p$-groups with the AEC
property and their direct products. In Section \ref{5} we will also
give a positive answer to Noether's problem for one series of $4$-generator
$p$-groups of nilpotency class $2$ that do not posses the AEC
property and that are not direct or central products of smaller groups.

%--------------------------------------Section 2-------------------------------------------------
\section{Preliminaries and notations}
\label{2}

Let $G$ be a group and $x,y\in G$. We define $x^y=y^{-1}xy$ and
write $[x,y]=x^{-1}x^y=x^{-1}y^{-1}xy$ for the commutator of $x$ and
$y$. We define the commutators of higher weight as
$[x_1,x_2,\dots,x_n]=[[x_1,\dots,x_{n-1}],x_n]$ for
$x_1,x_2,\dots,x_n\in G$.

Let $\varphi$ be an automorphism of $G$ and
$G^\varphi$ be an isomorphic copy of $G$ via $\varphi : x\mapsto
x^\varphi$. We define $\tau(G)$ to be the group generated by $G$ and
$G^\varphi$, subject to the following relations:
$[x,y^\varphi]^z=[x^z,(y^z)^\varphi]=[x,y^\varphi]^{z^\varphi}$ and
$[x,x^\varphi]=1$ for all $x, y, z\in G$. Obviously, the groups $G$
and $G^\varphi$ can be viewed as subgroups of $\tau(G)$. Let
$[G,G^\varphi]=\langle[x,y^\varphi] : x,y\in G\rangle$ be the
commutator subgroup.

Next, denote by $G\wedge G$ the nonabelian exterior square of $G$, which is a group generated by the
symbols $x\wedge y$ ($x,y\in G$), satisfying the relations
{\allowdisplaybreaks\begin{eqnarray*} &xy\wedge z&=(x^y\wedge
z^y)(y\wedge z),\\
&x\wedge yz&=(x\wedge z)(x^z\wedge y^z),\\
&x\wedge x&=1,
\end{eqnarray*}}
for all $x,y,z\in G$. Let
$[G,G]$ be the commutator subgroup of $G$. Observe that the
commutator map $\kappa : G\wedge G\to [G,G]$, given by $x\wedge
y\mapsto [x,y]$ is a well-defined group homomorphism. Let $M(G)$
denote the kernel of $\kappa$, and $M_0(G)$ detone the subgroup of
$M(G)$ generated by all $x\wedge y$ such that $x,y\in G$ commute.
Moravec proved in \cite{Mo1} that the Bogomolov multiplier $B_0(G)$ is (non-canonically)
isomorphic to the quotient group $M(G)/M_0(G)$.

Notice that the map $\phi : G\wedge G\to
[G,G^\varphi]$ given by $x\wedge y\mapsto [x,y^\varphi]$ is actually
an isomorphism of groups (see \cite{BM}).

Now, let $\kappa^*=\kappa\cdot\phi^{-1}$ be the composite map from
$[G,G^\varphi]$ to $[G,G]$, $M^*(G)=\ker\kappa^*$ and
$M_0^*(G)=\phi(M_0(G))$. Then $B_0(G)$ is clearly isomorphic to
$M^*(G)/M_0^*(G)$ by \cite{Mo1}. It is not hard to see that
\begin{equation*}
M^*(G)=\left\{\prod_{\text{finite}}[x_i,y_i^\varphi]^{\varepsilon_i}\in
[G,G^\varphi] : \varepsilon_i=\pm 1,
\prod_{\text{finite}}[x_i,y_i]^{\varepsilon_i}=1\right\},
\end{equation*}
and
\begin{equation*}
M_0^*(G)=\left\{\prod_{\text{finite}}[x_i,y_i^\varphi]^{\varepsilon_i}\in
[G,G^\varphi] : \varepsilon_i=\pm 1, [x_i,y_i]=1\right\}.
\end{equation*}
Thus, in order to prove that $B_0(G)=0$ for a given group $G$, it suffices
to show that $M^*(G)=M_0^*(G)$. This can be achieved by finding a
generating set of $M^*(G)$ consisting solely of elements of
$M_0^*(G)$.

The following two Lemmas contain various properties of $\tau(G)$ and $[G,G^\varphi]$ that will be used in our considerations.

\begin{lem}\label{l1}
{\rm (}\cite{BM}{\rm)} Let $G$ be a group.
\begin{enumerate}
    \item $[x,yz]=[x,z][x,y][x,y,z]$ and $[xy,z]=[x,z][x,z,y][y,z]$ for all $x,y,z\in G$.
    \item If $G$ is nilpotent of class $c$, then $\tau(G)$ is nilpotent of class at most
    $c+1$.
    \item If $G$ is nilpotent of class $\leq 2$, then $[G,G^\varphi]$ is abelian.
    \item $[x,y^\varphi]=[x^\varphi,y]$ for all $x,y\in G$.
    \item $[x,y,z^\varphi]=[x,y^\varphi,z]=[x^\varphi,y,z]=[x^\varphi,y^\varphi,z]=[x^\varphi,y,z^\varphi]=[x,y^\varphi,z^\varphi]$ for all $x,y,z\in G$.
    \item $[[x,y^\varphi],[a,b^\varphi]]=[[x,y],[a,b]^\varphi]$ for all $x,y,a,b\in G$.
    \item $[x^n,y^\varphi]=[x,y^\varphi]^n=[x,(y^\varphi)^n]$ for all integers $n$ and $x,y\in G$ with $[x,y]=1$.
    \item If $[G,G]$ is nilpotent of class $c$, then $[G,G^\varphi]$ is nilpotent of class $c$ or $c+1$.
\end{enumerate}
\end{lem}

\begin{lem}\label{l2}
{\rm (}\cite[Lemma 3.1]{Mo2}{\rm)} Let G be a nilpotent group of
class $\leq 3$. Then
\begin{equation*}
[x,y^n]=[x,y]^n[x,y,y]^{\binom{n}{2}}[x,y,y,y]^{\binom{n}{3}}
\end{equation*}
for all $x,y\in \tau(G)$ and every positive integer $n$.
\end{lem}

%--------------------------------------Section 3-------------------------------------------------
\section{Central products}
\label{3}

Let $E$ be a group, and let $N$ be a central subgroup of $E$. For
$G=E/N$ define a map $\eta:M^*(E)\to M^*(G)$ by
$\eta(\prod[x_i,y_i^\varphi]^{m_i})=\prod
[x_iN,(y_iN)^\varphi]^{m_i}$. Clearly, $\eta$ is a homomorphism.
Now, define a map $\xi:[G,G^\varphi]\to E$ by
$\xi([x,y^\varphi])=[\widetilde x,\widetilde y]$, where
$x=\widetilde xN,y=\widetilde yN$. Since $N\leq Z(G)$, the map $\xi$
does not depend on the choice of $\widetilde x$ and $\widetilde y$.
Moreover, $\xi$ is an epimorphism. Denote $N_1=\xi(M^*(G))$ and
$N_0=\xi(M_0^*(G))$. It is not hard to see that $N_1=N\cap [E,E]$
and $N_0=\langle [\widetilde x,\widetilde y]\in N\rangle$.

One can now easily verify that we have the following exact
sequences:
\begin{equation*}
M_0^*(E)\overset{\eta}{\longrightarrow}
M_0^*(G)\overset{\xi}{\longrightarrow} N_0
\end{equation*}
and
\begin{equation*}
M^*(E)\overset{\eta}{\longrightarrow}
M^*(G)\overset{\xi}{\longrightarrow} N_1.
\end{equation*}
Hence we obtain the exact sequence
\begin{equation}\label{B0}
B_0(E)\overset{\eta_*}{\longrightarrow}
B_0(G)\overset{\xi_*}{\longrightarrow} N_1/N_0,
\end{equation}
and in particular, the isomorphism $B_0(G)/\eta_*(B_0(E))\simeq
N_1/N_0$.

Next, we are going to develop further the case when $G$ is a central
product of two groups $G_1$ and $G_2$ with a common central
subgroup. Let $\theta:K_1\to K_2$ be an isomorphism, where $K_1\leq
Z(G_1)$ and $K_2\leq Z(G_2)$, and let $E=G_1\times G_2$. Then the
central product of $G_1$ and $G_2$ is defined as the quotient group
$G=E/N$, where $N=\{ab: a\in K_1,b\in K_2, \theta(a)=b^{-1}\}\in
Z(E)$.

\begin{thm}\label{main1}
Let $\theta:G_1\to G_2$ be a group homomorphism such that its
restriction $\theta\vert_{K_1}:K_1\to K_2$ is an isomorphism, where
$K_1\leq Z(G_1)$ and $K_2\leq Z(G_2)$. Let $G$ be a central product
of $G_1$ and $G_2$, i.e., $G=E/N$, where $E=G_1\times G_2$ and
$N=\{ab: a\in K_1,b\in K_2, \theta(a)=b^{-1}\}$. If
$B_0(G_1/K_1)=B_0(G_1)=B_0(G_2)=0$ then $B_0(G)=0$.
\end{thm}
\begin{proof}
Theorem \ref{t1.1} implies that $B_0(E)\simeq B_0(G_1)\times
B_0(G_2)=0$. Then from the exact sequence \eqref{B0} it follows that
$B_0(G)\simeq B_0(G)/\eta_*(B_0(E))\simeq N_1/N_0$. Therefore, we
need to show only that $N_0=N_1$.

Take arbitrary $\prod [\widetilde x_i,\widetilde y_i]^{m_i}=ab\in
N_1$, where $a\in K_1,b\in K_2,\theta(a)=b^{-1}, \widetilde
x_i,\widetilde y_i\in E$. We have now that $\widetilde
x_i=\alpha_{i1}\alpha_{i2}$ and $\widetilde
y_i=\beta_{i1}\beta_{i2}$ for $\alpha_{ij},\beta_{ij}\in G_j;1\leq
j\leq 2$. Since the elements of $G_1$ commute with the elements of
$G_2$, we obtain the formula $[\widetilde x_i,\widetilde
y_i]=[\alpha_{i1},\beta_{i1}][\alpha_{i2},\beta_{i2}]$. Hence
$a=\prod [\alpha_{i1},\beta_{i1}]^{m_i}$ and $b=\prod
[\alpha_{i2},\beta_{i2}]^{m_i}$.

Now, we may apply the exact sequence \eqref{B0} for the central
group extension
$$1\to K_1\to G_1\to G_1/K_1\to 1.$$
From $B_0(G_1/K_1)=0$ it follows that
$$\frac{K_1\cap [G_1,G_1]}{\mathcal
K(G_1)\cap K_1}\simeq \frac{B_0(G_1/K_1)}{\eta_*(B_0(G_1))}\simeq
0,$$ where $\mathcal K(G_1)$ is the set of commutators in $G_1$.
Therefore we may assume that $[\alpha_{i1},\beta_{i1}]$ are in $K_1$
for all $i$. Since $\theta$ is a homomorphism, we have that
$\theta(a)=\theta(\prod [\alpha_{i1},\beta_{i1}]^{m_i})=\prod
[\theta(\alpha_{i1}),\theta(\beta_{i1})]^{m_i}=b^{-1}$. Note that
the commutators $[\theta(\alpha_{i1}),\theta(\beta_{i1})]$ are in
$K_2$ for all $i$, so their position in the decomposition of $b$ is
of no importance. Therefore, {\allowdisplaybreaks\begin{eqnarray*}
&ab&=\prod [\widetilde x_i,\widetilde y_i]^{m_i}=\prod
[\alpha_{i1},\beta_{i1}]^{m_i}\prod
[\theta(\alpha_{i1}),\theta(\beta_{i1})]^{-m_i}\\
&&=\prod
([\alpha_{i1},\beta_{i1}][\theta(\alpha_{i1}),\theta(\beta_{i1})]^{-1})^{m_i}=\prod
[\alpha_{i1}\theta(\beta_{i1}),\beta_{i1}\theta(\alpha_{i1})]^{m_i}.
\end{eqnarray*}}
Since
$[\alpha_{i1}\theta(\beta_{i1}),\beta_{i1}\theta(\alpha_{i1})]=[\alpha_{i1},\beta_{i1}](\theta([\alpha_{i1},\beta_{i1}]))^{-1}\in
N$, we obtain finally that $ab\in N_0$.
\end{proof}

The first straightforward application of Theorem \ref{main1} is for
the extra-special $p$-groups. Recall that a $p$-group $G$ is
extra-special if its center $Z$ is cyclic of order $p$, and the
quotient $G/Z$ is a non-trivial elementary abelian $p$-group. Kang
and Kunyavski\u{i} also proved in another way the following.

\begin{cor}\label{c1}
{\rm (}\cite[Proposition 2.1]{KK}{\rm)} If $G$ is an extra-special
$p$-group of order $p^{2n+1}$ (for any $n\geq 1$), then $B_0(G)=0$.
\end{cor}
\begin{proof}
Note first that for any $n$ there are two extra-special $p$-groups
of order $p^{2n+1}$, and they are isoclinic. According to a recent
result by Moravec \cite{Mo3}, the isoclinic groups have isomorphic
Bogomolov multipliers. It is well known that one of these two groups
can be obtained as a central product of $n$ Heisenberg groups $H_3$
of order $p^3$. We apply induction by $n$. For $n=1$ we have
$B_0(H_3)=0$ (see \cite{Be}). Let $G$ be a central product of the
extra-special $p$-group $G_2$ of order $p^{2n-1}$ and $H_3$, where
we assume that $B_0(G_2)=0$. Denote by $K_1\cong C_p$ and $K_2\cong
C_p$ the centers of $H_3$ and $G_2$, respectively. It is easy to
define a monomorphism $\theta: H_3\to G_2$, such that
$\theta(K_1)=K_2$, and since $B_0(H_3/K_1)=B_0(C_p\times C_p)=0$, we
may apply Theorem \ref{main1}.
\end{proof}

Another application of Theorem \ref{main1} we find in the following
Corollary, where we consider the central product of two split
metacyclic $p$-groups.

\begin{cor}\label{c2}
Let
$G=\langle\beta_1,\beta_2,\alpha_1,\alpha_2:[\beta_1,\alpha_1]=\beta_1^{p^r},
[\beta_2,\alpha_2]=\beta_2^{p^r},\alpha_1^{p^{a_1}}=\alpha_2^{p^{a_2}}=\beta_1^{p^b}=\beta_2^{p^b}=1,\beta_1^{p^{b-1}}=\beta_2^{p^{b-1}}\rangle$,
where $a_1,a_2,1\leq r \leq b-1,2\leq b$. Then $B_0(G)=0$.
\end{cor}
\begin{proof}
For $1\leq i\leq 2$ define
$G_i=\langle\beta_i,\alpha_i:[\beta_i,\alpha_i]=\beta_i^{p^r},
\alpha_i^{p^{a_i}}=\beta_i^{p^b}=1\rangle$. Without loss of
generality we will assume that $a_1\geq a_2$. Define a map
$\theta:G_1\to G_2$ by $\theta(\alpha_1)=\alpha_2$ and
$\theta(\beta_1)=\beta_2$. We are going to show that $\theta$ is a
homomorphism. Indeed, any element from $G_1$ can be written in the
form $\beta_1^y\alpha_1^z$. We have that
$\theta((\beta_1^w\alpha_1^x)(\beta_1^y\alpha_1^z))=\theta(\beta_1^{w+y(1+p^r)^x}\alpha_1^{x+z})=\beta_2^{w+y(1+p^r)^x}\alpha_2^{x+z}=\theta(\beta_1^w\alpha_1^x)\theta(\beta_1^y\alpha_1^z)$.
Next, note that
$[\beta_i^{p^{b-1}},\alpha_i]=\beta_i^{p^{r+b-1}}=1$, so
$\beta_i^{p^{b-1}}\in Z(G_i)$. For $1\leq i\leq 2$ define
$K_i=\langle\beta_i^{p^{b-1}}\rangle\cong C_p$. Clearly the
restriction $\theta\vert_{K_1}:K_1\to K_2$ is an isomorphism, so $G$
is a central product of the metacyclic $p$-groups $G_1$ and $G_2$.
Note that $G_1/K_1$ is also metacyclic. Our result now follows from
Theorems \ref{t1.2} and \ref{main1}.
\end{proof}

The group $G$ given in the statement of Corollary \ref{c2} is a
$4$-generator $p$-group of nilpotency class $2$ that does not posses
the AEC property. In the following section we will consider another group with the same properties which is not a central product of smaller groups.

%--------------------------------------Section 5-------------------------------------------------
\section{Noether's problem for groups of nilpotency class $2$}
\label{5}

Let $K$ be any field. A field extension $L$ of $K$ is called
rational over $K$ (or $K$-rational, for short) if $L\simeq
K(x_1,\ldots,x_n)$ over $K$ for some integer $n$, with
$x_1,\ldots,x_n$ algebraically independent over $K$. Now let $G$ be
a finite group. Let $G$ act on the rational function field
$K(x(g):g\in G)$ by $K$ automorphisms defined by $g\cdot x(h)=x(gh)$
for any $g,h\in G$. Denote by $K(G)$ the fixed field $K(x(g):g\in
G)^G$. {\it Noether's problem} then asks whether $K(G)$ is rational
over $K$. This is related to the inverse Galois problem, to the
existence of generic $G$-Galois extensions over $K$, and to the
existence of versal $G$-torsors over $K$-rational field extensions
\cite[33.1, p.86]{Sw,Sa1,GMS}. Noether's problem for abelian groups
was studied extensively by Swan, Voskresenskii, Endo, Miyata and
Lenstra, etc. The reader is referred to Swan's paper for a survey of
this problem \cite{Sw}.

We list several results which will be used in the sequel.

\begin{thm}\label{t5.1}
{\rm (}\cite[Theorem 1]{HK}{\rm )} Let $G$ be a finite group acting
on $L(x_1,\dots,x_m)$, the rational function field of $m$ variables
over a field $L$ such that
\begin{description}
    \item [(i)] for any $\sigma\in G, \sigma(L)\subset L;$
    \item [(ii)] the restriction of the action of $G$ to $L$ is
    faithful;
    \item [(iii)] for any $\sigma\in G$,
    \begin{equation*}
\begin{pmatrix}
\sigma(x_1)\\
\vdots\\
\sigma(x_m)\\
\end{pmatrix}
=A(\sigma)\begin{pmatrix}
x_1\\
\vdots\\
x_m\\
\end{pmatrix}
+B(\sigma)
\end{equation*}
where $A(\sigma)\in\GL_m(L)$ and $B(\sigma)$ is $m\times 1$ matrix
over $L$. Then there exist $z_1,\dots,z_m\in L(x_1,\dots,x_m)$ so
that $L(x_1,\dots,x_m)^G=L^G(z_1,\dots,z_m)$ and $\sigma(z_i)=z_i$
for any $\sigma\in G$, any $1\leq i\leq m$.
\end{description}
\end{thm}

\begin{thm}\label{t5.2}
{\rm (}\cite[Theorem 3.1]{AHK}{\rm )} Let $G$ be a finite group
acting on $L(x)$, the rational function field of one variable over a
field $L$. Assume that, for any $\sigma\in G,\sigma(L)\subset L$ and
$\sigma(x)=a_\sigma x+b_\sigma$ for any $a_\sigma,b_\sigma\in L$
with $a_\sigma\ne 0$. Then $L(x)^G=L^G(z)$ for some $z\in L[x]$.
\end{thm}

\begin{thm}\label{t5.3}
{\rm (Kuniyoshi }\cite[Theorem 1.7]{CK}{\rm )} If $char K=p>0$ and
$G$ is a finite $p$-group, then $K(G)$ is rational over $K$.
\end{thm}

\begin{thm}\label{t5.4}
{\rm (Kang }\cite[Theorem 4.1]{Ka1}{\rm )} Let $G$ be a metacyclic
$p$-group of exponent $p^e$, generated by $\sigma$ and $\tau$ such
that $\sigma^{p^m}=\tau^{p^n}=1,\tau^{-1}\sigma\tau=\sigma^s$ for
$s=\varepsilon+ap^r$, where $a=1+b\cdot p^t$ for $t\in\mathbb N$ and
$b\in\{-1,0,1\}$. Let $K$ be a field, containing a primitive
$p^e$-th root of unity, and let $\zeta$ be a primitive $p^m$-th root
of unity. Then $K(u_0,u_1,\dots,u_{p^n-1})^G$ is rational over $K$,
where $G$ acts on $u_0,\dots,u_{p^n-1}$ by
\begin{eqnarray*}
\sigma&:&u_i\mapsto \zeta^{s^i}u_i,\\
\tau&:&u_0\mapsto u_1\mapsto\cdots\mapsto u_{p^n-1}\mapsto u_0.
\end{eqnarray*}
\end{thm}

\begin{lem}\label{l5.5}
{\rm (}\cite{Ka2,HuK}{\rm)} Let $\langle\tau\rangle$ be a cyclic
group of order $n>1$, acting on $L(v_1,\dots,v_{n-1})$, the rational
function field of $n-1$ variables over a field $L$ such that
\begin{eqnarray*}
\tau&:&v_1\mapsto v_2\mapsto\cdots\mapsto v_{n-1}\mapsto (v_1\cdots
v_{n-1})^{-1}\mapsto v_1.
\end{eqnarray*}
If $L$ contains a primitive $n$th root of unity $\xi$, then
$L(v_1,\dots,v_{n-1})=L(s_1,\dots,s_{n-1})$ where $\tau:s_i\mapsto
\xi^is_i$ for $1\leq i\leq n-1$.
\end{lem}

Now, let us consider $2$-generator $p$-groups of nilpotency class
$2$. They were recently classified by Ahmad, Magidin and Morse
\cite{AMM}, although we will not use their classification theorems. Let $G$ be any $2$-generator $p$-group of class $2$ of
order $p^n$. Then $[G,G]$ is a central subgroup of $G$ that is
isomorphic to $C_{p^c}$ with $c\geq 1$, and $G/[G,G]$ is isomorphic
to $C_{p^a}\times C_{p^b}$ with $n=a+b+c$. Without loss of
generality assume $a\geq b$. Let $\{\alpha,\beta\}$ be a transversal
of $G/[G,G]$. Then $\alpha^{p^a}$ and $\beta^{p^b}$ are elements of
$[G,G]$, and we have $[\alpha^{p^a},\beta]=[\alpha,\beta]^{p^a}=1$
and $[\alpha,\beta^{p^b}]=[\alpha,\beta]^{p^b}=1$. Hence $p^c$, the
order of $[G,G]$, divides both $p^a$ and $p^b$. It follows that
$1\leq c\leq b\leq a$. From this analysis we may view any
$2$-generator group $G$ of order $p^n$ and class $2$ as a central
extension of the form
\begin{equation*}
1\longrightarrow C_{p^c}=\langle[\beta,\alpha]\rangle\longrightarrow
 G\longrightarrow C_{p^a}\times C_{p^b} \longrightarrow 1.
\end{equation*}
Since $H=\langle[\beta,\alpha],\alpha\rangle$ is a normal abelian
subgroup of $G$, the group $G$ has the AEC property. Our
verification shows however, that this group does not fit in the
conditions of the known rationality criteria. We give an answer to
Noether's problem for $G$ in the following.

\begin{thm}
Let $n\geq 3$ and let $G$ be a $2$-generator $p$-group of class $2$
of order $p^n$. Let $G$ be of exponent $p^e$. Assume that $K$ is any
field satisfying that either {\rm (i)} char $K = p>0$, or {\rm (ii)}
char $K \ne p$ and $K$ contains a primitive $p^e$-th root of unity.
Then $K(G)$ is rational over $K$.
\end{thm}
\begin{proof}
Let $V$ be a $K$-vector space whose dual space $V^*$ is defined as
$V^*=\bigoplus_{g\in G}K\cdot x(g)$ where $G$ acts on $V^*$ by
$h\cdot x(g)=x(hg)$ for any $h,g\in G$. Thus $K(V)^G=K(x(g):g\in
G)^G=K(G)$. The key idea is to find a faithful $G$-subspace $W$ of
$V^*$ and to show that $W^G$ is rational over $K$. Then by Theorem
\ref{t5.1} it will follow that $K(G)$ rational over $K$.

Let $\gamma=[\beta,\alpha]$, and let $H=\langle\beta,\gamma\rangle$.
The subspace $W$ is obtained as an induced representation from $H$.
Without loss of generality we may assume that $\beta^{p^b}=1$, so
$H\simeq C_{p^b}\times C_{p^c}$

Define $X_1,X_2\in V^*$ by
\begin{equation*}
X_1=\sum_{i=0}^{p^b-1}x(\beta^i),~
X_2=\sum_{i=0}^{p^c-1}x(\gamma^i).
\end{equation*}
Note that $\beta\cdot X_1=X_1$ and $\gamma\cdot X_2=X_2$.

Let $\zeta_{p^b}\in K$ be a primitive $p^b$-th root of unity, and
let $\zeta_{p^c}\in K$ be a primitive $p^c$-th root of unity. Define
$Y_1,Y_2\in V^*$ by

\begin{equation*}
Y_1=\sum_{i=0}^{p^c-1}\zeta_{p^c}^{-1}\gamma^i\cdot X_1,~
Y_2=\sum_{i=0}^{p^b-1}\zeta_{p^b}^{-1}\beta^i\cdot X_2.
\end{equation*}

It follows that {\allowdisplaybreaks \begin{align*}
\beta\ :\ &Y_1\mapsto Y_1,~ Y_2\mapsto\zeta_{p^b} Y_2,\\
\gamma\ :\ &Y_1\mapsto\zeta_{p^c} Y_1,~ Y_2\mapsto Y_2.
\end{align*}}
Thus $K\cdot Y_1+K\cdot Y_2$ is a representation space of the
subgroup $H$. The induced subspase $W$ depends on the relations in
$G$.  It is well known that the case $\alpha^{p^a}=\beta^f\gamma^h$
can easily be reduced to the case $\alpha^{p^a}=1$ (see e.g.
\cite[Proof of Theorem 1.8, Step 2]{Mi1}).

Define $x_i=\alpha^i\cdot Y_1,y_i=\alpha^i\cdot Y_2$ for $0\leq
i\leq p^a-1$. Note that from the relation $[\beta,\alpha]=\gamma$ it
follows that $\beta\alpha^i=\alpha^i\beta\gamma^i$. We have now
{\allowdisplaybreaks \begin{align*}
\beta\ :\ &x_i\mapsto\zeta_{p^c}^i x_i,~ y_i\mapsto\zeta_{p^b}y_i,\\
\gamma\ :\ &x_i\mapsto\zeta_{p^c} x_i,~ y_i\mapsto y_i,\\
\alpha\ :\ &x_0\mapsto x_1\mapsto\cdots\mapsto x_{p^a-1}\mapsto x_0,\\
&y_0\mapsto y_1\mapsto\cdots\mapsto y_{p^a-1}\mapsto y_0.
\end{align*}}
for $0\leq i\leq p^a-1$. For $1\leq i\leq p^a-1$, define
$u_i=x_i/x_{i-1}$ and $v_i=y_i/y_{i-1}$. Thus $K(x_i,y_i:0\leq i\leq
p^a-1)=K(x_0,y_0,u_i,v_i:1\leq i\leq p^a-1)$ and for every $g\in G$
\begin{equation*}
g\cdot x_0\in K(u_i,v_i:1\leq i\leq p^a-1)\cdot x_0,~ g\cdot y_0\in
K(u_i,v_i:1\leq i\leq p^a-1)\cdot y_0,
\end{equation*}
while the subfield $K(u_i,v_i:1\leq i\leq p^a-1)$ is invariant by
the action of $G$. Thus $K(x_i,y_i:0\leq i\leq
p^a-1)^{G}=K(u_i,v_i:1\leq i\leq p^a-1)^{G}(u,v)$ for some $u,v$
such that $\alpha(v)=\beta(v)=\gamma(v)=v$ and
$\alpha(u)=\beta(u)=\gamma(u)=u$. We have now {\allowdisplaybreaks
\begin{align*}
\beta\ :\ &u_i\mapsto\zeta_{p^c} u_i,~ v_i\mapsto v_i,\\
\tag{5.1} \gamma\ :\ &u_i\mapsto u_i,~ v_i\mapsto v_i,\\
\alpha\ :\ &u_1\mapsto u_2\mapsto\cdots\mapsto u_{p^a-1}\mapsto (u_1u_2\cdots u_{p^a-1})^{-1},\\
&v_1\mapsto v_2\mapsto\cdots\mapsto v_{p^a-1}\mapsto (v_1v_2\cdots
v_{p^a-1})^{-1},
\end{align*}}
for $1\leq i\leq p^a-1$. From Theorem \ref{t5.2} it follows that if
$K(u_i,v_i:1\leq i\leq p^a-1)^{G}$ is rational over $K$, so is
$K(x_i,y_i:0\leq i\leq p^a-1)^{G}$ over $K$.

Now, consider the metacyclic $p$-group $\widetilde
G=\langle\sigma,\tau:\sigma^{p^{2c}}=\tau^{p^a}=1,\tau^{-1}\sigma\tau=\sigma^{k},k=1+p^c\rangle$.

Define $X=\sum_{0\leq j\leq
p^{2c}-1}\zeta_{p^{2c}}^{-j}x(\sigma^j),V_i=\tau^i X$ for $0\leq
i\leq p^a-1$. It follows that
\begin{eqnarray*}
\sigma&:&V_i\mapsto \zeta_{p^{2c}}^{k^i}V_i,\\
\tau&:&V_0\mapsto V_1\mapsto\cdots\mapsto V_{p^a-1}\mapsto V_0.
\end{eqnarray*}
Note that $K(V_0,V_1,\dots,V_{p^a-1})^{\widetilde G}$ is rational by
Theorem \ref{t5.4}.

Define $U_i=V_i/V_{i-1}$ for $1\leq i\leq p^a-1$. Then
$K(V_0,V_1,\dots,V_{p^a-1})^{\widetilde G}=K(U_1,U_2,\dots,$
$U_{p^a-1})^{\widetilde G}(U)$ where
\begin{eqnarray*}
\sigma&:&U_i\mapsto \zeta_{p^{2c}}^{k^i-k^{i-1}}U_i,~ U\mapsto U\\
\tau&:&U_1\mapsto U_2\mapsto\cdots\mapsto U_{p^a-1}\mapsto
(U_1U_2\cdots U_{p^a-1})^{-1},~ U\mapsto U.
\end{eqnarray*}

Notice that
$\zeta_{p^{2c}}^{k^i-k^{i-1}}=\zeta_{p^{2c}}^{(1+p^c)^{i-1}p^c}=\zeta_{p^{c}}^{(1+p^c)^{i-1}}=\zeta_{p^{c}}$.
Compare Formula (5.1) (i.e., the actions of $\beta$ and $\alpha$ on
$K(u_i:1\leq i\leq p^a-1)$) with the action of $\widetilde G$ on
$K(U_i:1\leq i\leq p^a-1)$. They are the same. Hence, according to
Theorem \ref{t5.4}, we get that $K(u_1,\dots,u_{p^a-1})^{G}(u)\cong
K(U_1,\dots,U_{p^a-1})^{\widetilde
G}(U)=K(V_0,V_1,\dots,V_{p^a-1})^{\widetilde G}$ is rational over
$K$. Since by Lemma \ref{l5.5} we can linearize the action of
$\alpha$ on $K(v_i:1\leq i\leq p^a-1)$, we obtain finally that
$K(u_i,v_i:1\leq i\leq p^a-1)^{\langle\beta,\alpha\rangle}$ is
rational over $K$.
\end{proof}

The next step is to investigate Noether's problem for groups of nilpotency class $2$ that do not have
the AEC property. It turns out that the existing methods can not answer this problem entirely. Indeed, results about Noether's problem for such groups are rarely seen in the literature. We managed to discover one series of $p$-groups for which we will give a positive answer to Noether's problem. For any natural number $r$ define
{\allowdisplaybreaks\begin{eqnarray}\label{G}
&&G_r=\langle\alpha_1,\alpha_2,\beta_1,\beta_2,\gamma_1,\gamma_2:[\beta_1,\alpha_1]=\gamma_1=\beta_1^{p^r},[\beta_2,\alpha_2]=\gamma_2=\beta_2^{p^r},[\alpha_1,\alpha_2]=\gamma_2,\\
&&\nonumber\quad\quad\quad\quad\quad\quad\quad\quad\quad\quad\quad\quad  \gamma_1^{p^r}=\gamma_2^{p^r}=\alpha_1^{p^r}=\alpha_2^{p^r}=1\rangle.
\end{eqnarray}}
where all the relations of the form $[x, y] = 1$ between the
generators have been omitted from the list. Thus $\gamma_1,\gamma_2\in Z(G_r)$, i.e., $G_r$ is a
$p$-group of nilpotency class $2$ .

By examining the relations in this group we see that it is well
defined. It is not hard to see also that $G_r$ is not a direct or a
central product of smaller groups. Moreover, $G$ does not
have the AEC property. First, let us verify that the Bogomolov multiplier of $G_r$ is trivial.
\begin{thm}
If $G$ is isomorphic to the group $G_r$ defined by the presentation \eqref{G}, then $B_0(G)=0$.
\end{thm}
\begin{proof}
The group $[G,G^\varphi]$ is generated
modulo $M_0^*(G)$ by
$[\beta_1,\alpha_1^\varphi],[\beta_2,\alpha_2^\varphi]$ and
$[\alpha_1,\alpha_2^\varphi]$. Every element $w\in
[G,G^\varphi]$ can be written as
$w=[\beta_1,\alpha_1^\varphi]^m[\beta_2,\alpha_2^\varphi]^n[\alpha_1,\alpha_2^\varphi]^q\tilde
w$, where $\tilde w\in M_0^*(G)$. This gives
$w^{\kappa^*}=\gamma_1^m\gamma_2^{n+q}$, therefore $w\in M^*(G)$
if and only if $p^r$ divides both $m$ and $n+q$. By Lemma \ref{l2}
we have that {\allowdisplaybreaks\begin{eqnarray*}
&1&=[\alpha_2^\varphi,\alpha_1^{p^r}]\\
&&=[\alpha_2^\varphi,\alpha_1]^{p^r}[\alpha_2^\varphi,\alpha_1,\alpha_1]^{\binom{p^r}{2}}[\alpha_2^\varphi,\alpha_1,\alpha_1,\alpha_1]^{\binom{p^r}{3}}\\
&&=[\alpha_2^\varphi,\beta_2]^{p^r}[\gamma_2^{-1},\alpha_1^\varphi]^{\binom{p^r}{2}}[\gamma^{-1},\alpha_1,\alpha_1^\varphi]^{\binom{p^r}{3}}\\
&&=[\alpha_2^\varphi,\alpha_1]^{p^r},
\end{eqnarray*}}
and similarly,
$[\alpha_1^\varphi,\beta_1]^{p^r},[\alpha_2^\varphi,\beta_2]^{p^r}\in
M_0^*(G)$. It follows that
\begin{equation*}
M^*(G)=\langle[\beta_2,\alpha_2^\varphi][\alpha_1,\alpha_2^\varphi]^{-1}\rangle
M_0^*(G).
\end{equation*}
Note that {\allowdisplaybreaks\begin{eqnarray*}
&&[\beta_2\alpha_2,\alpha_2\alpha_1]=[\beta_2,\alpha_1]^{\alpha_2}[\beta_2,\alpha_2]^{\alpha_1\alpha_2}[\alpha_2,\alpha_1][\alpha_2,\alpha_2]^{\alpha_1}=[\alpha_2,\gamma^{-1}]=1,
\end{eqnarray*}}
hence $[\beta_2\alpha_2,(\alpha_2\alpha_1)^\varphi]\in M_0^*(G)$.
Expanding the latter using the class restriction and Lemma \ref{l1},
we get {\allowdisplaybreaks\begin{eqnarray*}
&[\beta_2\alpha_2,(\alpha_2\alpha_1)^\varphi]&=[\beta_2,\alpha_1^\varphi]^{\alpha_2}[\beta_2,\alpha_2^\varphi]^{\alpha_1^\varphi\alpha_2}[\alpha_2,\alpha_1^\varphi][\alpha_2,\alpha_2^\varphi]^{\alpha_1^\varphi}\\
&&=[\beta_2,\alpha_1^\varphi]^{\alpha_2}[\beta_2,\alpha_2^\varphi][\beta_2,\alpha_2^\varphi,\alpha_1^\varphi\alpha_2][\alpha_2,\alpha_1^\varphi][\alpha_2,\alpha_2^\varphi]^{\alpha_1^\varphi}.
\end{eqnarray*}}
Observe that $[\beta_2,\alpha_1^\varphi]^{\alpha_2},
[\beta_2,\alpha_2^\varphi,\alpha_1^\varphi\alpha_2]$ and
$[\alpha_2,\alpha_2^\varphi]^{\alpha_1^\varphi}$ all belong to
$M_0^*(G)$. Thus we conclude that
$[\beta_2,\alpha_2^\varphi][\alpha_1,\alpha_2^\varphi]^{-1}=[\beta_2,\alpha_2^\varphi][\alpha_2,\alpha_1^\varphi]\in
M_0^*(G)$, as required.
\end{proof}

Finally, we will show that Noether's problem has a positive answer for $G_r$.

\begin{thm}
Let $G$ be isomorphic to the group $G_r$ defined by the presentation \eqref{G}. Assume that $K$ is any field satisfying that
either {\rm (i)} char $K = p>0$, or {\rm (ii)} char $K \ne p$ and
$K$ contains a primitive $p^{2r}$-th root of unity. Then $K(G)$ is
rational over $K$.
\end{thm}
\begin{proof}
Let $V$ be a $K$-vector space whose dual
space $V^*$ is defined as $V^*=\bigoplus_{g\in G}K\cdot x(g)$ where
$G$ acts on $V^*$ by $h\cdot x(g)=x(hg)$ for any $h,g\in G$. Thus
$K(V)^G=K(x(g):g\in G)^G=K(G)$.

Consider the subgroups $H_1=\langle\beta_1,\beta_2,\alpha_1\rangle$
and $H_2=\langle\beta_1,\beta_2,\alpha_2\rangle$ of $G$. Note that
$H_1=\langle\beta_2\rangle\times\langle\beta_1,\alpha_1\rangle\simeq
C_{p^{2r}}\times\langle\beta_1,\alpha_1\rangle$. Hence we get a
linear character of $H_1$ so that $\langle\beta_1,\alpha_1\rangle$
is in the kernel. Explicitly, we may define an action of $H_1$ on
$K\cdot X$ by
\begin{equation*}
\beta_2(X)=\zeta_{p^{2r}}X,~ \beta_1(X)=X,~ \alpha_1(X)=X.
\end{equation*}
Similarly,
$H_2=\langle\beta_1\rangle\times\langle\beta_2,\alpha_2\rangle\simeq
C_{p^{2r}}\times\langle\beta_1,\alpha_1\rangle$. Hence we get a
linear character of $H_2$ so that $\langle\beta_2,\alpha_2\rangle$
is in the kernel. Explicitly, we may define an action of $H_2$ on
$K\cdot Y$ by
\begin{equation*}
\beta_1(Y)=\zeta_{p^{2r}}Y,~ \beta_2(Y)=Y,~ \alpha_2(Y)=Y.
\end{equation*}
Construct the induced representations of these linear characters by
defining $x_i=\alpha_1^i\cdot Y,y_i=\alpha_2^i\cdot X$ for $0\leq
i\leq p^r-1$. Thus we get an action of $G$ on $(\bigoplus_{0\leq
i\leq p^r-1}K\cdot x_i)\oplus(\bigoplus_{0\leq i\leq p^r-1}K\cdot
y_i)$. Since
$\beta_j\alpha_j^i=\alpha_j^i\beta_j\gamma_j^i,\alpha_1\alpha_2^i=\alpha_2^i\alpha_1\gamma_2^i$
and $\alpha_2\alpha_1^i=\alpha_1^i\alpha_2\gamma_2^{-i}$ for $1\leq
j\leq 2,0\leq i\leq p^r-1$, the action of $G$ is given as follows.
{\allowdisplaybreaks
\begin{align*}
\beta_1\ :\ &x_i\mapsto\zeta_{p^{2r}}^{1+ip^r} x_i,~ y_i\mapsto y_i,\\
\beta_2\ :\ &x_i\mapsto x_i,~ y_i\mapsto\zeta_{p^{2r}}^{1+ip^r} y_i,\\
\alpha_1\ :\ &x_0\mapsto x_1\mapsto\cdots\mapsto x_{p^r-1}\mapsto x_0,~ y_i\mapsto\zeta_{p^r}^i y_i,\\
\alpha_2\ :\ &x_i\mapsto x_i,~ y_0\mapsto y_1\mapsto\cdots\mapsto
y_{p^r-1}\mapsto y_0,
\end{align*}}
for $0\leq i\leq p^r-1$. For $1\leq i\leq p^r-1$, define
$u_i=x_i/x_{i-1}$ and $v_i=y_i/y_{i-1}$. Thus $K(x_i,y_i:0\leq i\leq
p^r-1)=K(x_0,y_0,u_i,v_i:1\leq i\leq p^r-1)$ and for every $g\in G$
\begin{equation*}
g\cdot x_0\in K(u_i,v_i:1\leq i\leq p^r-1)\cdot x_0,~ g\cdot y_0\in
K(u_i,v_i:1\leq i\leq p^r-1)\cdot y_0,
\end{equation*}
while the subfield $K(u_i,v_i:1\leq i\leq p^r-1)$ is invariant by
the action of $G$. Thus $K(x_i,y_i:0\leq i\leq
p^r-1)^{G}=K(u_i,v_i:1\leq i\leq p^r-1)^{G}(u,v)$ for some $u,v$
such that $\beta_1(v)=\beta_2(v)=\alpha_1(v)=\alpha_2(v)=v$ and
$\beta_1(u)=\beta_2(u)=\alpha_1(u)=\alpha_2(u)=u$. We have now
{\allowdisplaybreaks \begin{align*}
\beta_1\ :\ &u_i\mapsto\zeta_{p^r} u_i,~ v_i\mapsto v_i,\\
\beta_2\ :\ &u_i\mapsto u_i,~ v_i\mapsto\zeta_{p^r} v_i,\\
\alpha_1\ :\ &u_1\mapsto u_2\mapsto\cdots\mapsto u_{p^r-1}\mapsto (u_1u_2\cdots u_{p^r-1})^{-1},~ v_i\mapsto\zeta_{p^r} v_i,\\
\alpha_2\ :\ &u_i\mapsto u_i,~ v_1\mapsto v_2\mapsto\cdots\mapsto
v_{p^r-1}\mapsto (v_1v_2\cdots v_{p^r-1})^{-1},
\end{align*}}
for $1\leq i\leq p^r-1$. From Theorem \ref{t5.2} it follows that if
$K(u_i,v_i:1\leq i\leq p^r-1)^{G}$ is rational over $K$, so is
$K(x_i,y_i:0\leq i\leq p^r-1)^{G}$ over $K$.

Define $w_1=v_1^{p^r}$ and $w_i=v_i/v_{i-1}$ for $2\leq i\leq
p^r-1$. Then $K(u_i,w_i:1\leq i\leq p^r-1)=K(u_i,v_i:1\leq i\leq
p^r-1)^{\langle\beta_2\rangle}$ and the action of $\alpha_2$ on $w_i
(1\leq i\leq p^r-1)$ is {\allowdisplaybreaks
\begin{align*}
\alpha_2\ :\ &w_1\mapsto w_1w_2^{p^r},\\
&w_2\mapsto w_3\mapsto\cdots\mapsto w_{p^r-1}\mapsto
(w_1w_2^{p^r-1}w_3^{p^r-2}\cdots w_{p^r-1}^2)^{-1}\mapsto\\
&\mapsto w_1w_2^{p^r-2}w_3^{p^r-3}\cdots w_{p^r-2}^2w_{p^r-1}\mapsto
w_2.
\end{align*}}
Define $z_1=w_2,z_i=\alpha_2^i\cdot w_2$ for $2\leq i\leq p^r-1$.
Now the action of $\alpha_2$ is {\allowdisplaybreaks
\begin{align*}
\alpha_2\ :\ &z_1\mapsto z_2\mapsto\cdots\mapsto z_{p^r-1}\mapsto
(z_1z_2\cdots z_{p^r-1})^{-1}.
\end{align*}}
Since $w_1=(z_{p^r-1}z_1^{p^r-1}z_2^{p^r-2}\cdots
z_{p^r-2}^2)^{-1}$, we get that
$K(w_1,\dots,w_{p^r-1})=K(z_1,\dots,z_{p^r-1})$. From Lemma
\ref{l5.5} it follows that the action of $\alpha_2$ on
$K(z_1,\dots,z_{p^2-1})$ can be linearized. Moreover, the actions of
$\beta_1$ and $\alpha_1$ on $K(z_1,\dots,z_{p^2-1})$ are trivial. In
this way we reduce the rationality problem of $K(u_i,v_i:1\leq i\leq
p^r-1)^{G}$ to the rationality of $K(u_i:1\leq i\leq
p^r-1)^{\langle\beta_1,\alpha_1\rangle}$ over $K$. It remains to
repeat the same argument for $\beta_1$ and $\alpha_1$ to complete
the proof.
\end{proof}


\begin{thebibliography}{AA}
\bibitem{AHK}
H. Ahmad, S. Hajja and M. Kang, Rationality of some projective
linear actions, {\it J. Algebra} {\bf 228} (2000), 643--658.
\bibitem{AMM}
A. Ahmad, A. Magidin and R. Morse, Two-generator $p$-groups of
nilpotency class two and their conjugacy classes, {\it Publ. Math.
Debrecen} {\bf 81} (2012), 145--166.
\bibitem{Be}
E. Beneish, Stable rationality of certain invariant field, {\it J.
Algebra} {\bf 269} (2003), 373--380.
\bibitem{Bo}
F. A. Bogomolov, The Brauer group of quotient spaces by linear group
actions, {\it Math. USSR Izv.} {\bf 30} (1988), 455–-485.
\bibitem{BM}
R. Blyth and R. Morse, Computing the nonabelian tensor square of
polycyclic groups, {\it J. Algebra} {\bf 321} (2009), 2139--2148.
\bibitem{CK}
H. Chu and M. Kang, Rationality of $p$-group actions, {\it J.
Algebra} {\bf 237} (2001), 673--690.
\bibitem{CM}
Y. Chen, R. Ma, Bogomolov multipliers of some groups of order
$p^{6}$, (preprint available at arXiv:1302.0584v2 [math.AG).
\bibitem{GMS}
S. Garibaldi, A. Merkurjev and J-P. Serre, ``Cohomological
invariants in Galois cohomology'', AMS Univ. Lecture Series vol. 28,
Amer. Math. Soc., Providence, 2003.
\bibitem{HK}
S. Hajja and M. Kang, Some actions of symmetric groups, {\it J.
Algebra} {\bf 177} (1995), 511--535.
\bibitem{HoK}
A. Hoshi and M. Kang, Unramified Brauer groups for groups of order
$p^5$, (preprint available at arXiv:1109.2966v1 [math.AC]).
\bibitem{HuK}
S. J. Hu and M. Kang, Noether's problem for some $p$-groups, in
``Cohomological and geometric approaches to rationality problems'',
edited by F. Bogomolov and Y. Tschinkel, Progress in Math. vol. 282,
Birkh\"auser, Boston, 2010.
\bibitem{HKK}
A. Hoshi, M. Kang, B. Kunyavskii, Noether's problem and unramified
Brauer groups, Asian J. Math. 17, No. 4, 689--714 (2013).
\bibitem{Ka1}
M. Kang, Noether's problem for metacyclic $p$-groups, {\it Adv.
Math.} {\bf 203} (2005), 554--567.
\bibitem{Ka2}
M. Kang, Noether's problem for $p$-groups with a cyclic subgroup of
index $p^2$, {\it Adv. Math.} {\bf 226} (2011) 218--234.
\bibitem{Ka3}
M. Kang, Bogomolov multipliers and retract rationality for
semi-direct products, {\it J. Algebra} 397, 407--425 (2014).
\bibitem{KK}
M. Kang, B. Kunyavski\u{i}, The Bogomolov multiplier of rigid finite
groups, {\it Archiv der Mathematik}, {\bf 102} (3), 2014, pp 209-218.
\bibitem{Ku}
B. E. Kunyavski\u{i}, The Bogomolov multiplier of finite simple
groups, {\it Cohomological and geometric approaches to rationality
problems}, 209–217, Progr. Math., 282, Birkh\"auser Boston, Inc.,
Boston, MA, 2010.
\bibitem{Mi1}
I. Michailov, Noether's problem for abelian extensions of cyclic
$p$-groups, {\it Pacific J. Math}, {\bf 270} (1), 2014, p. 167-189.
\bibitem{Mo1}
P. Moravec, Unramified Brauer groups of finite and infinite groups,
{\it Amer. J. Math.} {\bf 134} (2012), 1679--1704.
\bibitem{Mo2}
P. Moravec, Groups of order $p^5$ and their unramified Brauer
groups, {\it J. Algebra} {\bf 372}, 420--427 (2012).
\bibitem{Mo3}
P. Moravec, Unramified Brauer groups and isoclinism, {\it Ars Math. Contemp.} {\bf 7} (2) 2014, 337--340.
\bibitem{Sa1}
D. J. Saltman, Generic Galois extensions and problems in field
theory, {\it Adv. Math.} {\bf 43} (1982), 250--283.
\bibitem{Sa2}
D. J. Saltman, Noether's problem over an algebraically closed field,
{\it Invent. Math.} {\bf 77}  (1984), 71--84.
\bibitem{Sh}
I. R. \u{S}afarevi\u{c}, The L\"uroth problem, {\it Proc. Steklov
Inst. Math.} {\bf 183} (1991), 241-–246.
\bibitem{Sw}
R. Swan, Noether's problem in Galois theory, in ``Emmy Noether in
Bryn Mawr'', edited by B. Srinivasan and J. Sally, Springer-Verlag,
Berlin, 1983.
\end{thebibliography}
\end{document}